\documentclass{compositio}

\title{Exact functors on perverse coherent sheaves}
\author{Clemens Koppensteiner}
\email{clemens@math.northwestern.edu}
\address{Department of Mathematics, Northwestern University, 2033 Sheridan Rd, Evanston, IL 60208, USA}
\classification{14F05 (primary), 18E30, 14B15 (secondary)}
\keywords{perverse coherent sheaves, local cohomology}

\usepackage[T1]{fontenc}
\usepackage[utf8]{inputenc}

\usepackage{csquotes}
\usepackage{amsmath,amsfonts,amssymb,amsthm, mathtools}
\usepackage[backend=biber,
    maxnames=100,
    style=alphabetic
    ]{biblatex}
\usepackage{microtype}

\usepackage[colorlinks=false,unicode=true]{hyperref}

\theoremstyle{plain}
\newtheorem{Thm}{Theorem}
\newtheorem{Prop}[Thm]{Proposition}
\newtheorem{Lem}[Thm]{Lemma}

\theoremstyle{definition}
\newtheorem{Def}[Thm]{Definition}

\theoremstyle{remark}
\newtheorem{Rem}[Thm]{Remark}
\newtheorem{Ex}[Thm]{Example}

\newcommand\sheaf{\mathcal}
\newcommand\liealg[1]{\mathfrak{#1}}
\newcommand\liesl[1]{\liealg{sl}_#1}
\newcommand\SL[1]{\mathrm{SL}_#1}
\newcommand\sO{\sheaf{O}}
\newcommand\cat{\mathbf}
\newcommand\catModules[1]{#1\text{-}\cat{Mod}}
\newcommand\catCoh[2][]{\cat{Coh}^{#1}(#2)}
\newcommand\sheafHom{\underline{\operatorname{Hom}}}
\newcommand\supp{\operatorname{supp}}
\newcommand\res[2]{\mathchoice{\left.#1\right|_{#2}}{#1|_{#2}}{#1|_{#2}}{#1|_{#2}}}
\newcommand\codim{\operatorname{codim}}
\DeclareMathOperator\Spec{Spec}
\newcommand\ideal\mathfrak
\newcommand\setset\mathfrak
\newcommand\perv[1][p]{\mkern1mu{}^{#1}\mkern-3mu}
\newcommand\dualize{\mathbb D}
\newcommand\lc[1]{\Gamma_{\mkern-3mu#1}}
\newcommand\measuringFam{\mathfrak M}
\newcommand\Xtop[1][X]{#1^{G,\mathrm{top}}}

\makeatletter
\renewcommand\p@enumii{}
\renewcommand\theenumi{(\@roman\c@enumi)}
\makeatother

\begin{document}

\begin{abstract}
    Inspired by symplectic geometry and a microlocal characterizations of perverse (constructible) sheaves we consider an alternative definition of perverse coherent sheaves.
    We show that a coherent sheaf is perverse if and only if $R\lc Z\sheaf F$ is concentrated in degree $0$ for special subvarieties $Z$ of $X$.
    These subvarieties $Z$ are analogs of Lagrangians in the symplectic case.
\end{abstract}

\maketitle

\section{Introduction}

A general way to obtain insights about the heart of a t-structure is to study exact functors on the t-structure.
For example, for the category of constructible perverse sheaves on a complex manifold, one can obtain a large amount of exact functors by taking vanishing cycles \cite[Corollary~10.3.13]{KashiwaraSchapira:1994:SheavesOnManifolds}.

Let $\sheaf F$ be a constructible (middle) perverse sheaf on an affine Kähler manifold $X$.
Let $x \in  X$ be point and $f\colon X \to  \mathbb{C}$ a suitably chosen holomorphic Morse function with $f(x) = 0$ and single critical point $x$.
Then the stalk $(\phi _{\!f}\sheaf F)_x$ is concentrated in cohomological degree $0$ (here we use $\phi _{\!f}$ for the vanishing cycles functor).
A more \enquote{geometric} formulation of this statement can be obtained in the following way.
Let $L$ be the stable manifold for the gradient of the Morse function $\Re f$.
Write $i_x \colon \{x\} \hookrightarrow L$ and $i_L\colon L \hookrightarrow X$ for the inclusions.
Then $i_x^*i_L^! \sheaf F$ is also concentrated in cohomological degree $0$.
Note that $L$ is a Lagrangian with respect to the symplectic structure given by the Kähler form.

Inspired by this result in the constructible setting, we prove a related statement in the setting of coherent sheaves.
Let $X$ be a variety with an action by an algebraic group $G$ with finitely many orbits.
In this situation one can define \emph{perverse t-structures} on the derived category $D_c^b(X)^G$ of coherent $G$-equivariant sheaves (see Section~\ref{sec:Kashiwara} for a review of the theory of perverse coherent sheaves).
For a given perversity function we define the notion of a \emph{measuring subvariety} (Definition~\ref{def:measuring}) as an analog of the Lagrangian $L$ in the constructible case.
Our main theorem (Theorem~\ref{thm:main}) then states that a coherent sheaf $\sheaf F \in  D^b_c(X)^G$ is perverse if and only if $i_Z^!\sheaf F$ is concentrated in cohomological degree $0$ for sufficiently many measuring subvarieties $Z$ of $X$.

\begin{Ex}
    Let $N$ be the nilpotent cone in the complex Lie algebra $\liesl n$ and let $G = \SL n$ act on $N$ adjointly.
    Then the dimensions of the $G$-orbits in $N$ are known to be even-dimensional.
    Thus there exists a middle perversity $p$ with $p(O) = \frac12 \dim O$ for each $G$-orbit $O$.
    It is know that any maximal nilpotent subalgebra $\liealg n$ intersects each $G$-orbit closure in $N$ in a Lagrangian.
    So $\liealg n$ fulfills the first condition for being a measuring subvariety.

    However, as far as we know the second, technical, condition on a measuring subvariety (i.e.~that $\liealg n$ is a set-theoretic local complete intersection in each orbit closure) is still an open problem for general $n$.
    If it is shown to hold, then the theorem implies that a sheaf $\sheaf F \in  D_c^b(N)^G$ is perverse if and only if $R\lc{\liealg n}\sheaf F$ is concentrated in degree $0$.
\end{Ex}

We expect that in general interesting examples of measuring subvarieties are provided by Lagrangians on (singular) symplectic varieties in the sense of \cite{Beauville:2000:SymplecticSingularities}.
Here by a Lagrangian we mean a closed subvariety that intersects every symplectic leaf in a Lagrangian.
This gives a clear analogy with the constructible setting, indicating that a kind of \enquote{microlocal characterization} of perverse coherent sheaves might be possible.
However, while a thorough examination of the symplectic case would be interesting, we do not provide one in the present paper.

Finally, since the motivating observation about constructible perverse sheaves does not seem to be in the literature (though \cite[Theorem~3.5]{MirkovicVilonen:2007:GLdualityRepresentations} is in the same spirit), we give a direct proof of the statement in the appendix.

\subsection{Setup and notation}

Let $X$ be a Noetherian separated scheme of finite type over an algebraically closed field $k$.
Let $G$ be an algebraic group over $k$ acting on $X$.
Until Section~\ref{sec:measuring} we include the possibility of $G$ being trivial.
We write $\Xtop$ for the subset of the Zariski space of $X$ consisting of generic points of $G$-invariant subschemes and equip $\Xtop$ with the induced topology.
To simplify notation, if $x \in  \Xtop$ is any point, we write $\overline x$ for the closure $\overline{\{x\}}$ and set $\dim x = \dim \overline x$.

We write $D(X)$, $D_{qc}(X)$ and $D_c(X)$ for the derived category of $\sO_X$-modules and its full subcategories consisting of complexes with quasi-coherent and coherent cohomology sheaves respectively.
The corresponding categories of $G$-equivariant sheaves (i.e.\ the categories for the quotient stack $[X/G]$) are denoted $D(X)^G$, $D_{qc}(X)^G$ and $D_c(X)^G$.
As usual, $D^b(X)$ (etc.) is the full subcategory of $D(X)$ consisting of complexes with cohomology in only finitely many degrees.
All functors are derived, though we usually do not explicitly mention it in the notation.

For a subset $Y$ of a topological space $X$ we write $i_Y$ for the inclusion of $Y$ into $X$. 
If $x \in  X$ is a point, then we simply write $i_x$ for $i_{\{x\}}$.
Let $Z$ be a closed subset of $X$.
For an $\sO_X$-module $\sheaf F$ let $\lc Z\sheaf F$ be the subsheaf of $\sheaf F$ of sections with support in $Z$ \cite[Variation~3 in IV.1]{Hartshorne:1966:ResiduesAndDuality}.
By abuse of notation, we simply write $\lc Z$ for the right-derived functor $R\lc Z\colon D_{qc}(X) \to  D_{qc}(X)$.
Recall that $\lc Z$ only depends on the closed subset $Z$, and not on the structure of $Z$ as a subscheme.

Let $x$ be a (not necessarily closed) point of $X$ and $\sheaf F \in  D^b(X)$.
Then $\mathbf i_x^*\sheaf F = \sheaf F_x \in  D^b(\catModules{\sO_x})$ denotes the (derived) functor of talking stalks.
We further set $\mathbf i_x^!\sheaf F = \mathbf i_x^*\lc {\overline x}$, cf.~\cite[Variation~8 in IV.1]{Hartshorne:1966:ResiduesAndDuality}.

We assume that $X$ has a $G$-equivariant dualizing complex $\sheaf R$ (see~\cite[Definition~1]{Bezrukavnikov:arXiv:PerverseCoherentSheaves}) which we assume to be normalized, i.e.\ $\mathbf i_x^! \sheaf R$ is concentrated in degree $-\dim x$ for all $x \in  \Xtop$.
For $\sheaf F \in  D(X)$ (or $D(X)^G$) we write $\dualize \sheaf F = \sheafHom_{\sO_X}(\sheaf F,\sheaf R)$ for its dual.

\subsection{Acknowledgements}

The motivating observation about constructible sheaves was suggested to me by my advisor, David Nadler.
I would also like to thank Sam Gunningham for valuable comments and Geordie Williamson for answering a question on MathOverflow (see Lemma~\ref{lem:constructible_local_vanishing}).
Finally, I am grateful for the helpful comments and suggestions by the anonymous referee.

\section{Perverse coherent sheaves}
\label{sec:Kashiwara}%

By a \emph{perversity} we mean a function $p\colon \{0,\dotsc,\dim X\} \to  \mathbb{Z}$.
For $x \in  \Xtop$ we abuse notation and set $p(x) = p(\dim x)$.
Then $p\colon \Xtop \to  \mathbb{Z}$ is a perversity function in the sense of~\cite{Bezrukavnikov:arXiv:PerverseCoherentSheaves}.
Note that we insist that $p(x)$ only depends on the dimension of $\overline x$.
A perversity is called \emph{monotone} if it is decreasing and \emph{comonotone} if the \emph{dual perversity} $\overline p(n) = -n - p(n)$ is decreasing.
It is \emph{strictly monotone} (resp.~\emph{strictly comonotone}) if for all $x,y \in  \Xtop$ with $\dim x < \dim y$ one has $p(x) > p(y)$ (resp.~$\overline p(x) > \overline p(y)$).
Note that a strictly monotone perversity is not necessarily strictly decreasing (e.g.~if $X$ only has even-dimensional $G$-orbits).

Recall that if $p$ is a monotone and comonotone perversity then Bezrukavnikov (following Deligne) defines a t-structure on $D_c^b(X)^G$ by taking the following full subcategories \cite{Bezrukavnikov:arXiv:PerverseCoherentSheaves,ArinkinBezrukavnikov:2010:PerverseCoherentSheaves}:
\begin{align*}
    \perv[p] D^{\leq 0}(X)^G & = 
    \bigl\{ \sheaf F \in  D_c^b(X)^G : \mathbf i_x^*\sheaf F \in  D^{\leq p(x)}(\catModules{\sO_x}) \text{ for all $x \in  \Xtop$}\bigr\}, \\
    \perv[p] D^{\geq 0}(X)^G & = 
    \bigl\{ \sheaf F \in  D_c^b(X)^G : \mathbf i_x^!\sheaf F \in  D^{\geq p(x)}(\catModules{\sO_x}) \text{ for all $x \in  \Xtop$}\bigr\}.
\end{align*}
The heart of this t-structure is called the category of \emph{perverse sheaves} with respect to the perversity $p$.

In~\cite{Kashiwara:2004:tStructureOnHolonomicDModuleCoherentOModules}, Kashiwara also gives a definition of a perverse t-structure on $D^b_{c}(X)$.
While we work in Bezrukavnikov's setting (i.e.\ in the equivariant derived category on a potentially singular scheme), we need a description of the perverse t-structure that is closer to the one Kashiwara uses.
This is accomplished in the following proposition.

\begin{Prop}\label{prop:equivDeligneKashiwara}%
    Let $\sheaf F \in  D_c^b(X)^G$ and let $p$ be a monotone and comonotone perversity function.
    \begin{enumerate}
        \item \label{li:prop:equivDeligneKashiwara:le}%
            The following are equivalent:
            \begin{enumerate}
                \item \label{li:prop:equivDeligneKashiwara:le:1}%
                    $\sheaf F \in  \perv D^{\leq 0}(X)^G$, i.e.\ $\mathbf i_x^*\sheaf F \in  D^{\leq p(x)}(\catModules{\sO_x})$ for all $x \in  \Xtop$;
                \item \label{li:prop:equivDeligneKashiwara:le:2}%
                    $p(\dim \supp H^{k}(\sheaf F)) \geq  k$ for all $k$.
            \end{enumerate}
        \item \label{li:prop:equivDeligneKashiwara:ge}%
            If $p$ is strictly monotone, then the following are equivalent
            \begin{enumerate}
                \item \label{li:prop:equivDeligneKashiwara:ge:1}%
                    $\sheaf F \in  \perv D^{\ge 0}(X)^G$, i.e.\ $\mathbf i_x^!\sheaf F \in  D^{\geq p(x)}(\catModules{\sO_x})$ for all $x \in  \Xtop$;
                \item \label{li:prop:equivDeligneKashiwara:ge:2}%
                    $\lc {\overline x}\sheaf F \in  D^{\geq p(x)}(X)$ for all $x \in  \Xtop$;
                \item \label{li:prop:equivDeligneKashiwara:ge:3}%
                    $\lc {Y}\sheaf F \in  D^{\geq p(\dim Y)}(X)$ for all $G$-invariant closed subvarieties $Y$ of $X$;
                \item \label{li:prop:equivDeligneKashiwara:ge:4}%
                    $\dim\left( \overline x \cap  \supp\left( H^k(\dualize \sheaf F) \right) \right) \leq  -p(x) - k$ for all $x \in  \Xtop$ and all $k$.
            \end{enumerate}
    \end{enumerate}
\end{Prop}

A crucial fact that we will implicitly use quite often in the following arguments is that the support of a coherent sheaf is always closed.
In particular, this means that if $x$ is a generic point and $\sheaf F$ a coherent sheaf, then $\mathbf i_x^* \sheaf F = 0$ if and only if $\res{\sheaf F}U = 0$ for some open set $U$ intersecting $\overline x$.

\begin{proof}\leavevmode
    \begin{enumerate}
        \item
            First let $\sheaf F \in  \perv D^{\leq 0}(X)^G$ and assume for contradiction that there exists an integer $k$ such that $p(\dim \supp H^{k}(\sheaf F)) < k$.
            Let $x$ be the generic point of an irreducible component of maximal dimension of $\supp H^{k}(\sheaf F)$.
            Then $H^k(\mathbf i_x^* \sheaf F) \ne 0$. 
            But on the other hand, $\mathbf i_x^*\sheaf F \in  D^{\leq p(x)}(\catModules{\sO_x})$ and $p(x) = p(\dim \supp H^{k}(\sheaf F)) < k$, yielding a contradiction.

            Conversely assume that $p(\dim \supp H^{k}(\sheaf F)) \geq  k$ for all $k$ and let $x \in  \Xtop$.
            If $H^k(\mathbf i_x^*\sheaf F) \ne 0$, then $\dim x \leq  \dim \supp H^{k}(\sheaf F)$.
            Thus monotonicity of the perversity implies that $\sheaf F \in  \perv D^{\leq 0}(X)^G$.
        \item
            The implications from \ref{li:prop:equivDeligneKashiwara:ge:3} to \ref{li:prop:equivDeligneKashiwara:ge:2} and \ref{li:prop:equivDeligneKashiwara:ge:2} to \ref{li:prop:equivDeligneKashiwara:ge:1} are trivial and the equivalence of \ref{li:prop:equivDeligneKashiwara:ge:2} and \ref{li:prop:equivDeligneKashiwara:ge:4} follows from Lemma~\ref{lem:supportAndLocalCohomology+} below.
            Thus we only need to show that \ref{li:prop:equivDeligneKashiwara:ge:1} implies \ref{li:prop:equivDeligneKashiwara:ge:3}.
            So assume that $\sheaf F \in  \perv D^{\geq 0}(X)^G$.
            We induct on the dimension of $Y$.
            
            If $\dim Y = 0$, then $\Gamma (X,\lc Y \sheaf F) = \bigoplus_{y \in  \Xtop[Y]} \mathbf i_y^!\sheaf F$ and thus $\lc Y\sheaf F \in  D^{\geq p(0)}(X)$ by assumption.

            Now let $\dim Y > 0$.
            We first assume that $Y$ is irreducible with generic point $x \in  \Xtop$.
            Let $k$ be the smallest integer such that $H^k(\lc {\overline x} \sheaf F) \ne 0$ and assume that $k < p(x)$.
            We will show that this implies that $H^k(\lc {\overline x}\sheaf F) = 0$, giving a contradiction.

            We first show that $H^k(\lc {\overline x}\sheaf F)$ is coherent.
            Let $j\colon X \setminus {\overline x} \hookrightarrow X$ and consider the distinguished triangle
            \[
                \lc {\overline x} \sheaf F \to  \sheaf F \to  j_*j^* \sheaf F \xrightarrow{+1}.
            \]
            Applying cohomology to it we get an exact sequence
            \[
                H^{k-1}(j_*j^*\sheaf F) \to  H^k(\lc{\overline x} \sheaf F) \to  H^k(\sheaf F).
            \]
            By assumption, $k-1 \le p(x) - 2$, so that $H^{k-1}(j_*j^*\sheaf F)$ is coherent by the Grothendieck Finiteness Theorem in the form of \cite[Corollary~3]{Bezrukavnikov:arXiv:PerverseCoherentSheaves}.
            As $H^k(\sheaf F)$ is coherent by definition, this implies that $H^k(\lc{\overline x} \sheaf F)$ also has to be coherent.

            Set $Z = \supp H^k(\lc {\overline x}\sheaf F)$.
            Then, since $i_x^* H^k(\lc {\overline x}\sheaf F) = H^k(\mathbf i_x^! \sheaf F)$ vanishes, $Z$ is a proper closed subset of $\overline x$.
            We consider the distinguished triangle
            \[
                H^k(\lc {\overline x}\sheaf F)[-k] \to 
                \lc {\overline x}\sheaf F \to 
                \tau _{>k}\lc {\overline x}\sheaf F \xrightarrow{+1},
            \]
            and apply $\lc Z$ to it:
            \[
                \lc Z H^k(\lc {\overline x}\sheaf F)[-k] =
                H^k(\lc {\overline x}\sheaf F)[-k] \to 
                \lc Z \sheaf F \to 
                \lc Z \tau _{>k}\lc {\overline x}\sheaf F \xrightarrow{+1}.
            \]
            Since $\dim Z < \dim x$, we can use the induction hypothesis and monotonicity of $p$ to deduce that $\lc Z \sheaf F$ is in degrees at least $p(\dim Z) \ge p(x) > k$.
            Clearly $\lc Z \tau _{>k}\lc {\overline x}\sheaf F$ is also in degrees larger than $k$.
            Hence $H^k(\lc {\overline x}\sheaf F)$ has to vanish.

            If $Y$ is not irreducible, let $Y_1$ be an irreducible component of $Y$ and $Y_2$ be the union of the other components.
            Then there is a Mayer-Vietoris distinguished triangle
            \[
                \lc {Y_1\cap Y_2} \sheaf F \to  \lc {Y_1} \sheaf F \oplus \lc{Y_2}\sheaf F \to  \lc{Y} \sheaf F \xrightarrow{+1},
            \]
            where $\lc {Y_1\cap Y_2} \sheaf F \in  D^{\ge p(\dim Y_1\cap Y_2)}(X) \subseteq D^{\ge p(\dim Y)+1}$ (by the induction hypothesis and strict monotonicity of $p$) and $\lc{Y_1} \sheaf F$ and $\lc{Y_2} \sheaf F$ are in $D^{\ge p(\dim Y)}(X)$ by induction on the number of components of $Y$.
            Thus $\lc Y \sheaf F \in  D^{\ge p(\dim Y)}$ as required.
            \qedhere
    \end{enumerate}
\end{proof}

\begin{Lem}[{\cite[Proposition~5.2]{Kashiwara:2004:tStructureOnHolonomicDModuleCoherentOModules}}]%
    \label{lem:supportAndLocalCohomology+}%
    Let $\sheaf F \in  D_c^b(X)$, $Z$ a closed subset of $X$, and $n$ an integer.
    Then $\lc Z\sheaf F \in  D_{qc}^{\geq n}(X)$ if and only if $\dim(Z\cap \supp(H^k(\dualize \sheaf F))) \le - k - n$ for all $k$.
\end{Lem}

This lemma extends \cite[Proposition~5.2]{Kashiwara:2004:tStructureOnHolonomicDModuleCoherentOModules} to singular varieties.
The proof is same as for the smooth case, but we will include it here for completeness.

\begin{proof}
    By~\cite[Proposition~VII.1.2]{SGA2}, $\lc Z\sheaf F \in  D_{qc}^{\ge n}(X)$ if and only if 
    \begin{equation}
        \label{eq:supportAndCohomology+}%
        \sheafHom(\sheaf G,\sheaf F) \in  D_c^{\ge n}(X)
    \end{equation}
    for all $\sheaf G \in  \catCoh{X}$ with $\supp \sheaf G \subseteq Z$.
    Let $d(n) = -n$ be the dual standard perversity.
    Then by~\cite[Lemma~5a]{Bezrukavnikov:arXiv:PerverseCoherentSheaves}, \eqref{eq:supportAndCohomology+} holds if and only if $\dualize \sheafHom(\sheaf G,\sheaf F) \in  \perv[d] D^{\le -n}(X)$.
    By~\cite[Proposition~V.2.6]{Hartshorne:1966:ResiduesAndDuality}, $\dualize \sheafHom(\sheaf G,\sheaf F) = \sheaf G \otimes_{\sO_X} \dualize \sheaf F$, so that by Proposition~\ref{prop:equivDeligneKashiwara}\ref{li:prop:equivDeligneKashiwara:le} we need to show that
    \[
        \dim \supp H^{k}\left(\sheaf G \otimes_{\sO_X} \dualize \sheaf F\right) \le - k - n 
    \]
    for all $k$.
    By~\cite[Lemma~5.3]{Kashiwara:2004:tStructureOnHolonomicDModuleCoherentOModules} (whose proof does not use the smoothness assumption) this is equivalent to
    \[
        \dim \left(Z\cap \supp H^{k}(\dualize \sheaf F)\right) \le - k - n
    \]
    for all $k$, completing the proof.
\end{proof}

\section{Measuring subvarieties}
\label{sec:measuring}%

From now on we assume that the $G$-action has finitely many orbits.

\begin{Def}\label{def:measuring}%
    Let $p$ be a perversity.
    A \emph{$p$-measuring subvariety} of $X$ is a closed subvariety $Z$ of $X$ such that the following conditions hold for each $x \in  \Xtop$ with $\overline x \cap  Z \ne \emptyset$:
    \begin{itemize}
        \item $\dim(\overline x \cap  Z) = p(x) + \dim x$;
        \item $\overline x \cap  Z$ is a set-theoretic local complete intersection in $\overline x$, i.e. up to radical $\overline x \cap  Z$ is locally defined in $\overline x$ by exactly $-p(x)$ functions.
    \end{itemize}

    A \emph{$p$-measuring family of subvarieties} is a collection $\measuringFam$ of $p$-measuring subvarieties of $X$ such that for each $x \in  \Xtop$ there exists $Z \in  \measuringFam$ with $Z \cap  \overline x \ne \emptyset$.
    We say that $X$ has \emph{enough $p$-measuring subvarieties} if there exists a $p$-measuring family of subvarieties of $X$.
\end{Def}

\begin{Rem}\label{rem:existence}%
    Let $Z$ be a $p$-measuring subvariety.
    Then $\dim(\overline x \cap  Z) = -\overline p(x)$.
    Thus comonotonicity of $p$ ensures that if $\dim y \leq  \dim x$ then $\dim (\overline y \cap  Z) \leq  \dim (\overline x \cap  Z)$ for each $p$-measuring $Z$.
    Monotonicity of $p$ then further says that $\dim (\overline x \cap  Z) - \dim (\overline y \cap  Z) \leq  \dim x - \dim y$.
    We clearly have $0 \le \dim(\overline x \cap  Z) \le \dim x$ and hence $-\dim x \le p(x) \le 0$.
    We will show in Theorem~\ref{thm:existance} that this condition is actually sufficient for the existence of enough $p$-measuring subvarieties, at least when $X$ is affine.
\end{Rem}

\begin{Thm}\label{thm:main}%
    Let $p$ be a strictly monotone and (not necessarily strictly) comonotone perversity and assume that $X$ has enough $p$-measuring subvarieties.
    Let $\measuringFam$ be a $p$-measuring family of subvarieties of $X$.
    Then we have:
    \begin{enumerate}
        \item $\perv[p] D^{\leq 0}(X)^G = \left\{ \sheaf F \in  D_c^b(X)^G : \lc Z\sheaf F \in  D^{\leq 0}(X) \text{ for all $Z \in  \measuringFam$}\right\}$;
        \item $\perv[p] D^{\geq 0}(X)^G = \left\{ \sheaf F \in  D_c^b(X)^G : \lc Z\sheaf F \in  D^{\geq 0}(X) \text{ for all $Z \in  \measuringFam$}\right\}$.
    \end{enumerate}
    Therefore the sheaf $\sheaf F \in  D_c^b(X)^G$ is perverse with respect to $p$ if and only if $\lc Z\sheaf F$ is cohomologically concentrated in degree $0$ for each $p$-measuring subvariety $Z \in  \measuringFam$.
\end{Thm}

The following lemma encapsulates the central argument of the proof of the first part of the theorem.

\begin{Lem}\label{lem:supportAndLocalCohomology-}%
    Let $\sheaf F \in  \catCoh{X}^G$ be a $G$-equivariant coherent sheaf on $X$, let $p$ be a monotone perversity and let $n$ be an integer.
    Assume that $X$ has enough $p$-measuring subvarieties and let $\measuringFam$ be a $p$-measuring family of subvarieties of $X$.
    Then the following are equivalent:
    \begin{enumerate}
        \item $p(\dim \supp \sheaf F) \geq  n$;
        \item \label{li:lem:supportAndLocalCohomology-:2}%
            $H^i(\lc Z\sheaf F) = 0$ for all $i \geq  -n+1$ and all $Z \in  \measuringFam$.
    \end{enumerate}
\end{Lem}

\begin{proof}
    Since $\supp \sheaf F$ is always a union of the closure of orbits, we can restrict to the support and assume that $\supp \sheaf F = X$.

    First assume that $p(\dim X) = p(\dim \supp \sheaf F) \geq  n$.
    Using a Mayer-Vietoris argument it suffices to check condition \ref{li:lem:supportAndLocalCohomology-:2} in the case that $X$ is irreducible.
    By the definition of a $p$-measuring subvariety and monotonicity of $p$, this implies that, up to radical, $Z$ can be locally defined by at most $-n$ equations.
    Thus $H^i(\lc Z\sheaf F) = 0$ for $i > -n$ \cite[Theorem~3.3.1]{BrodmannSharp:1998:LocalCohomology}. 

    Now assume conversely that $H^i(\lc Z\sheaf F) = 0$ for all $i \geq  -n+1$ and all measuring subvarieties $Z \in  \measuringFam$.
    We have to show that $p(\dim X) \geq  n$.
    Set $d = \dim X$.
    Choose any $p$-measuring subvariety $Z \in  \measuringFam$ that intersects a maximal component of $X$ non-trivially.
    Then $\codim_X Z = -p(d)$.
    We will show that $H^{-p(d)}(\lc Z \sheaf F) \ne 0$ and hence $p(d) \ge n$ by assumption.
    Take some affine open subset $U$ of $X$ such that $U \cap Z$ is non-empty, irreducible and of codimension $-p(d)$ in $U$.
    It suffices to show that the cohomology is non-zero in $U$.
    Thus we can assume without loss of generality that $X$ is affine, say $X = \Spec A$, and $Z$ is irreducible.
    Write $Z = V(\ideal p)$ for some prime ideal $\ideal p$ of $A$.
    By flat base change~\cite[Theorem~4.3.2]{BrodmannSharp:1998:LocalCohomology},
    \[
    \Gamma (X,H^{-p(d)}(\lc Z \sheaf F))_{\ideal p} = 
    \left(H_{\ideal p}^{-p(d)}(\Gamma (X,\sheaf F))\right)_{\ideal p} =
    H_{\ideal p_{\ideal p}}^{-p(d)}(\Gamma (X,\sheaf F)_{\ideal p})
    \]
    Since $\dim \supp \sheaf F = \dim X = d$, the dimension of the $A_{\ideal p}$-module $\Gamma (X,\sheaf F)_{\ideal p}$ is $-p(d)$.
    Thus by the Grothendieck non-vanishing theorem~\cite[Theorem~6.1.4]{BrodmannSharp:1998:LocalCohomology}
    %\cite[Théorème~V.3.1]{SGA2}
    we have $H_{\ideal p_{\ideal p}}^{-p(d)}(\Gamma (X,\sheaf F)_{\ideal p}) \ne 0$ and hence $\Gamma (X,H^{-p(d)}(\lc Z \sheaf F)) \ne 0$ as required.
\end{proof}

\begin{proof}[Proof of Theorem~\ref{thm:main}]\leavevmode
\begin{enumerate}
\item 
    We use the description of $\perv[p] D^{\leq 0}(X)^G$ given by Proposition~\ref{prop:equivDeligneKashiwara}, i.e.
    \[
    \perv D^{\leq 0}(X)^G = \left\{ \sheaf F \in  D_c^b(X)^G : p\left(\dim\left( \supp H^{n}(\sheaf F) \right)\right) \geq  n \text{ for all $n$}\right\}.
    \]
    We induct on the largest $k$ such that $H^k(\sheaf F) \ne 0$ to show that $\sheaf F \in  \perv D^{\leq 0}(X)^G$ if and only if $\lc Z\sheaf F \in  D^{\leq 0}(X)$ for all $p$-measuring subvarieties $Z \in  \measuringFam$.

    The equivalence is trivial for $k \ll 0$.
    For the induction step note that there is a distinguished triangle
    \[
    \tau _{<k} \sheaf F \to  \sheaf F \to  H^k(\sheaf F)[-k] \xrightarrow{+1}.
    \]
    Applying the functor $\lc Z$ and taking cohomology we obtain an exact sequence
    \begin{multline*}
        \cdots \to 
        H^1(\lc Z(\tau _{<k} \sheaf F)) \to 
        H^1(\lc Z\sheaf F) \to 
        H^{k+1}(\lc Z(H^k(\sheaf F))) \to  \\
        H^2(\lc Z(\tau _{<k} \sheaf F)) \to 
        H^2(\lc Z\sheaf F) \to 
        H^{k+2}(\lc Z(H^k(\sheaf F))) \to 
        \cdots.
    \end{multline*}
    By induction, $H^j(\lc Z(\tau _{<k} \sheaf F))$ vanishes for $j \geq  1$ so that $H^j(\lc Z\sheaf F) \cong H^{k+j}(\lc Z(H^k(\sheaf F)))$ for $j \geq  1$.
    Thus the statement follows from Lemma~\ref{lem:supportAndLocalCohomology-}.
\item 
    By Proposition~\ref{prop:equivDeligneKashiwara}\ref{li:prop:equivDeligneKashiwara:ge}, $\sheaf F \in  \perv D^{\geq 0}(X)^G$ if and only if
    \begin{align}
        \label{eq:main:+supp1}%
        & \dim\left( \overline x \cap  \supp\left( H^k(\dualize F) \right) \right) \leq  -p(x) - k &&  \text{for all $x \in  \Xtop$ and all $k$}. \\
        \intertext{Using Lemma~\ref{lem:supportAndLocalCohomology+} for $\lc Z\sheaf F \in  D^{\geq 0}(X)$, we see that we have to show the equivalence of \eqref{eq:main:+supp1} with}
        \notag
        & \dim\left( Z \cap  \supp\left( H^k(\dualize F) \right) \right) \leq  - k && \text{ for all $k$ and all $Z \in  \measuringFam$}. \\
        \intertext{Since there are only finitely many orbits, this is in turn equivalent to}
        \label{eq:main:+supp2}%
        & \dim\left( Z \cap  \overline x \cap  \supp\left( H^k(\dualize F) \right) \right) \leq  - k && \text{ $\forall\, x \in  \Xtop$, $k$ and $Z \in  \measuringFam$}.
    \end{align}
    We will show the equivalence for each fixed $k$ separately.
    Let us first show the implication from \eqref{eq:main:+supp1} to \eqref{eq:main:+supp2}.
    Since $H^k(\dualize \sheaf F)$ is $G$-equivariant and there are only finitely many $G$-orbits, it suffices to show \eqref{eq:main:+supp2} assuming that $\dim x \le \dim \supp H^k(\dualize F)$ and $\overline x \cap \supp H^k(\dualize F) \ne \emptyset$.
    Then $\dim\left(\overline x \cap  \supp\left( H^k(\dualize F) \right)\right) = \dim \overline x$.
    Thus,
    \begin{multline*}
        \dim\left(Z \cap  \overline x \cap  \supp\left( H^k(\dualize F) \right) \right) \le
        \dim(Z \cap  \overline x) =
        p(x) + \dim x = \\
        p(x) + \dim\left(\overline x \cap  \supp\left( H^k(\dualize F) \right)\right) \le
        p(x) - p(x) - k
        = -k.
    \end{multline*}
    
    Conversely, assume that \eqref{eq:main:+supp2} holds for $k$.
    If $\overline x \cap \supp H^k(\dualize F) = \emptyset$, then \eqref{eq:main:+supp1} is trivially true.
    Otherwise choose a $p$-measuring $Z$ that intersects $\supp H^k(\dualize F)$.
    First assume that $\overline x$ is contained in $\supp H^k(\dualize F)$.
    Then
    \begin{multline*}
        \dim\left(\overline x \cap  \supp\left( H^k(\dualize F) \right)\right) =
        \dim x =
        -p(x) + \dim(Z \cap  \overline x) = \\
        -p(x) + \dim\left(Z \cap  \overline x \cap  \supp\left( H^k(\dualize F) \right) \right) \le
        -p(x) - k.
    \end{multline*}
    Otherwise $\overline x \cap  \supp\left( H^k(\dualize F) \right) = \overline y$ for some $y \in  \Xtop$ with $\dim y < \dim x$.
    Then \eqref{eq:main:+supp1} holds for $y$ in place of $x$ and hence
    \[
    \dim\left( \overline x \cap  \supp\left( H^k(\dualize F) \right) \right) =
    \dim\left( \overline y \cap  \supp\left( H^k(\dualize F) \right) \right) \leq 
    -p(y) - k \leq 
    -p(x) - k
    \]
    by monotonicity of $p$.
    \qedhere
\end{enumerate}
\end{proof}

\begin{Ex}
    For the dual standard perversity $p(n) = -n$ (i.e.\ $p(x) = -\dim x$), we recover the definition of Cohen-Macaulay sheaves~\cite[Section~IV.3]{Hartshorne:1966:ResiduesAndDuality}.
\end{Ex}

Of course, for the theorem to have any content, one needs to show that $X$ has enough $p$-measuring subvarieties.
The next theorem shows that at least for affine varieties there are always enough measuring subvarieties whenever $p$ satisfies the obvious conditions (see Remark~\ref{rem:existence}).

\begin{Thm}\label{thm:existance}%
    Assume that $X$ is affine and the perversity $p$ is monotone and comonotone and satisfies $-n \le p(n) \le 0$ for $n \in  \{0,\dotsc,\dim X\}$.
    Then $X$ has enough $p$-measuring subvarieties.
\end{Thm}

\begin{proof}
    Let $X = \Spec A$.
    We induct on the dimension $d$.
    More precisely, we induct on the following statement.
    \begin{quote}
        There exists a closed subvariety $Z_d$ of $X$ such that for all $x \in  \Xtop$ the following hold:
        \begin{itemize}
            \item $Z_d \cap \overline x \ne \emptyset$ and $Z_d \cap \overline x$ is a set-theoretic local complete intersection in $\overline x$.
            \item If $\dim x \le d$, then $\dim(\overline x \cap  Z_d) = p(x) + \dim x$.
            \item If $\dim x > d$, then $\dim(\overline x \cap  Z_d) = p(d) + \dim x$.
        \end{itemize}
    \end{quote}
    We set $p(-1) = 0$.
    The statement is trivially true for $d = -1$, e.g.~take $Z = X$.
    Assume that the statement is true for some $d \ge -1$.
    We want to show it for $d+1 \le \dim X$.

    If $p(d) = p(d+1)$, then $Z_{d+1} = Z_{d}$ works.
    Otherwise, by (co)monotonicity, $p(d+1) = p(d) - 1$.
    Set $S = \bigcup \{ \overline x \in  \Xtop : \dim x \le d\}$.
    Since there are only finitely many orbits, we can choose a function $f$ such that $f$ vanishes identically on $S$, $V(f)$ does not share a component with $Z_d$ and $V(f)$ intersects every $\overline x$ with $\dim x > d$.
    Then $Z_{d+1} = Z_d \cap V(f)$ satisfies the conditions.
\end{proof}

\section*{Appendix. Constructible sheaves}

We return now to the claim about exact functors on the t-structure of constructible perverse sheaves made in the introduction.
Let $X$ be a complex manifold and $\setset S$ a finite stratification of $X$ by complex submanifolds.
We write $D^b_{\setset S}(X)$ for the bounded derived category of $\setset S$-constructible sheaves on $X$.
We call a sheaf $\sheaf F \in  \smash[t]{D^b_{\setset S}(X)}$ perverse if it is perverse with respect to the middle perversity function on $\setset S$.
We are going to formulate and prove an analog of Theorem~\ref{thm:main} in this situation.

A closed real submanifold $Z$ of $X$ is called a \emph{measuring submanifold} if for each stratum $S$ of $X$ either $Z \cap  \overline S = \emptyset$ or $\dim_\mathbb{R} Z \cap  S = \dim_\mathbb{C} S$.
A \emph{measuring family} is a collection of measuring submanifolds $\{ Y_i \}$ such that each connected component of each stratum has non-empty intersection with at least one $Y_i$.
Similarly to Theorem~\ref{thm:existance}, one shows inductively that such a collection of submanifolds always exists.

\begin{Thm}
    Let $\measuringFam$ be a measuring family of submanifolds of $X$.
    A sheaf $\sheaf F \in  D^b_{\setset S}(X)$ is perverse if and only if $i_Z^! \sheaf F$ is concentrated in cohomological degree $0$ for each submanifold $Z \in  \measuringFam$.
\end{Thm}

The proof of the following lemma is based on a MathOverflow post by Geordie Williamson \cite{MO:VanishingShriekRestrictionConstructible}.
The author takes responsibility for possible mistakes.

\begin{Lem}\label{lem:constructible_local_vanishing}%
    Let $X$ be a real manifold, $\sheaf F$ be a constructible sheaf (concentrated in degree $0$) on $X$ and let $i\colon Z \hookrightarrow X$ be the inclusion of a closed submanifold.
    Then $H^j(i^!\sheaf F) = 0$ for $j>\codim_XZ$.
\end{Lem}

\begin{proof}
    By taking normal slices we can reduce to the case that $Z = \{z\}$ is a point.
    Let $j$ be the inclusion of $X \setminus \{z\}$ into $X$ and consider the distinguished triangle
    \[
        i_!i^! \sheaf F \to  \sheaf F \to  j_*j^*\sheaf F \xrightarrow{+1}.
    \]
    By~\cite[Lemma~8.4.7]{KashiwaraSchapira:1994:SheavesOnManifolds} we have
    \[
        H^j(j_*j^*\sheaf F)_z = H^j(S^{\dim X - 1}_\epsilon , \sheaf F)
    \]
    for a sphere $S^{\dim X - 1}_\epsilon $ around $z$ of sufficiently small radius.
    The latter cohomology vanishes for $j \ge \dim X$ and hence $H^j(i^! \sheaf F) = 0$ for $j>\dim X$ as required.
\end{proof}

\begin{proof}[Proof of Theorem]
    Define two full subcategories $\perv[L] D^{\le 0}(X)$ and $\perv[L] D^{\ge 0}(X)$ of $D^b_{\setset S}(X)$ by
    \begin{align*}
        \perv[L] D^{\leq 0}(X) & = \left\{ \sheaf F \in  D^b_{\setset S}(X) : i_Z^! \sheaf F \in  D^{\leq 0}(Z) \text{ for all $Z \in  \measuringFam$} \right\}, \\
        \perv[L] D^{\geq 0}(X) & = \left\{ \sheaf F \in  D^b_{\setset S}(X) : i_Z^! \sheaf F \in  D^{\geq 0}(Z) \text{ for all $Z \in  \measuringFam$} \right\}.
    \end{align*}
    We will show that these categories are the same as the categories $\perv D^{\leq 0}(X)$ and $\perv D^{\geq 0}(X)$ defining the perverse t-structure on $D^b_{\setset S}(X)$.

    We induct on the number of strata.
    If $X$ consists of only one stratum and $Z$ is a measuring submanifold, then $i_Z^! \sheaf F \cong \omega _{Z/X} \otimes i_Z^* \sheaf F$ and hence $i_Z^! \sheaf F$ is in degree $0$ if and only if $\sheaf F$ is in degree $-\frac 12 \dim_\mathbb{R} X$.
    So assume that $X$ has more then one stratum.
    Without loss of generality we can assume that $X$ is connected.
    Let $U$ be the union of all open strata and $F$ its complement.
    Both $U$ and $F$ are unions of strata of $X$.
    Let $j$ be the inclusion of $U$ and $i$ the inclusion of $X$. 
    
    \begin{itemize}
        \item 
            If $\sheaf F \in  \perv D^{\leq 0}(X)$, then $\sheaf F \in  \perv[L] D^{\leq 0}$ follows in exactly the same way as in the coherent case, using Lemma~\ref{lem:constructible_local_vanishing}.
        \item 
            Let $\sheaf F \in  \perv D^{\geq 0}(X)$.
            Then $i^! \sheaf F \in  \perv D^{\geq 0}(F)$ and $j^*\sheaf F \in  \perv D^{\geq 0}(U)$.
            Let $Z$ be a measuring subvariety.
            Consider the distinguished triangle 
            \[ 
                i_*i^! \sheaf F \to  \sheaf F \to  j_*j^*\sheaf F.
            \]
            Using base change, induction and the (left)-exactness of the push-forward functors one sees that $i_Z^!$ of the outer sheaves in the triangle are concentrated in non-negative degrees.
            Thus so is $i_Z^! \sheaf F$.
        \item 
            Let $\sheaf F \in  \perv[L] D^{\geq 0}(X)$.
            Since all measurements are local this implies that $j^* \sheaf F \in  \perv[L] D^{\geq 0}(U) = \perv D^{\geq 0}(U)$.
            Using the same triangle and argument as in the last point, this implies that also $i^! \sheaf F \in  \perv[L] D^{\geq 0}(F) = \perv D^{\geq 0}(F)$.
            Hence, by recollement, $\sheaf F \in  \perv D^{\geq 0}(X)$.
        \item 
            Finally, let $\sheaf F \in  \perv[L] D^{\leq 0}(X)$.
            Again this immediately implies that $j^* \sheaf F \in  \perv[L] D^{\leq 0}(U) = \perv D^{\leq 0}(U)$.
            Thus $j_!j^*\sheaf F \in  \perv D^{\leq 0}(X)$.
            Let $Z$ be a measuring submanifold and consider the distinguished triangle
            \[
                i_Z^! j_!j^*\sheaf F \to  i_Z^! \sheaf F \to  i_Z^! i_*i^* \sheaf F.
            \]
            By what we already know, the first sheaf is concentrated in non-positive degrees and hence so is $i_Z^! i_*i^* \sheaf F$.
            By base change and the exactness of $i_*$ this implies that $i^* \sheaf F \in  \perv[L] D^{\leq 0}(F) = \perv D^{\leq 0}(F)$.
            Hence, by recollement, $\sheaf F \in  \perv D^{\leq 0}(X)$.
            \qedhere
    \end{itemize}
\end{proof}

\begin{Rem}
    The equality $\perv D^{\geq 0}(X) = \perv[L]D^{\geq 0}(X)$ could also be proved in exactly the same way as in the coherent case, using~\cite[Exercise~X.10]{KashiwaraSchapira:1994:SheavesOnManifolds}.
\end{Rem}

\printbibliography

\end{document}